\documentclass{article}

\usepackage[english]{babel}
\usepackage[utf8x]{inputenc}
\usepackage[T1]{fontenc}

\usepackage{amsmath, amsfonts, amssymb, amsthm}
\usepackage{graphicx}
\usepackage[colorinlistoftodos]{todonotes}
\usepackage[colorlinks=true, allcolors=blue]{hyperref}

\newtheorem{thm}{Theorem}
\newtheorem{lem}[thm]{Lemma}
\newtheorem{prop}[thm]{Proposition}

\newtheorem*{main}{Main Theorem}

\DeclareMathOperator{\conv}{conv}
\DeclareMathOperator{\cone}{cone}
\DeclareMathOperator{\lin}{lin}
\DeclareMathOperator{\mult}{mult}

\begin{document}

\title{Explicit computations of Fourier transforms of polyhedral cones}
\author{Quang-Nhat LE
\thanks{This project has received funding from the European Research Council (ERC) under the European Union's Horizon 2020 research and innovation programme (grant agreement No [ERC StG 716424 - CASe]).}
}

\maketitle

\begin{abstract}
The Fourier transforms of polyhedral cones can be used, via Brion’s theorem, to compute various geometric quantities of polytopes, such as volumes, moments, and lattice-point counts. We present a novel method of computing these conic Fourier transforms by polynomial interpolation. Given the fact that computing volumes of polytopes is \#P-hard (Dyer--Frieze \cite{dyer1988complexity}), we cannot hope for efficient algorithms in the general case. However, with extra assumptions on the combinatorics of the cone, we demonstrate it is possible to compute its Fourier transform efficiently.
\end{abstract}

\section{Introduction}

Fourier analysis is a marvelous tool to tackle problems in polyhedral geometry. It can be used to study continuous quantities such as volumes (Postnikov \cite{postnikov2009permutohedra}), moments (Brion--Vergne \cite{brion1997residue}) and polynomial integration (Barvinok \cite{barvinok1992exponential}. It has also been employed to investigate discrete volumes which include Ehrhart functions (Diaz--Robins \cite{diaz1997ehrhart}, Barvinok--Pommersheim \cite{barvinok1999algorithmic}), solid-angle sums (DeSario--Robins \cite{desario2011generalized}, Diaz--Le--Robins \cite{diaz1602fourier}) and exponential sums (Barvinok \cite{barvinok1993computing}). In a lot of the above use cases, the Fourier transform of a polytope is a central object. One common way to analyze this object is to apply Brion's theorem to decompose this polyhedral Fourier transform into a sum of the Fourier transforms of the tangent cones at the vertices. The following general version was proved by Alexander Barvinok (1992) \cite{barvinok1992exponential}
\begin{thm}[Brion's theorem]
	For any convex polytope $P \subset \mathbb{R}^d$, we have the decomposition
    \begin{equation}\label{eq:simplicial}
    	\hat{1}_P (\xi) = \sum_{v \text{ vertex of } P} \hat{1}_{K_P(v)} (\xi),
    \end{equation}
    The \textbf{Fourier transform} is defined as $\hat{f} (\xi) := \int_{\mathbb{R}^d} f(x) e^{2 \pi i \langle x, \xi \rangle} dx$. Next, the \textbf{indicator function} of a set $S$ is $1_S (x) := 1$ if $x \in S$ and $0$ if $x \notin S$. Finally, the \textbf{tangent cone} at the vertex $x$ is defined by
    \begin{equation}
    	K_P(v) := \{ v + x : v + tx \in P \text{ for some } t > 0 \}.
    \end{equation}
\end{thm}
Therefore, the Fourier transforms of the cones $K_P(v)$ can be very useful in the study of the geometry of the polytope $P$. This is our main object of study.

Let $K \subset \mathbb{R}^d$ be a strictly convex (i.e. $K$ does not contain a line) polyhedral cone with apex $v$ and the set of generators $W := \{w_1, \dots, w_n\}$, with $n \geq d$. By \textbf{diagonals} of $K$, we mean $(d-1)$-subsets of $W$. A diagonal is called \textbf{extremal} if it is contained in the boundary $\partial K$ of $K$, and \textbf{interior} otherwise.

If $K$ is simplicial, that is $n = d$, it is very straightforward to calculate
\begin{equation}
	\hat{1}_K (\xi) = \frac{|\det(w_1, \dots, w_d)|}{\langle w_1, \xi \rangle \dots \langle w_d, \xi \rangle} \cdot e^{-2 \pi i \langle v, \xi \rangle}.
\end{equation}

When $K$ is not necessarily simplicial, we can triangulate $K$ and sum up the Fourier transforms of the simplicial faces. Thus, we can see that there exists a homogeneous polynomial $p_K$ of degree $n-d$ such that
\begin{equation}\label{eqdef:pK}
	\hat{1}_K (\xi) = \frac{p_K(\xi)}{\prod_{i=1}^n \langle w_i, \xi \rangle} \cdot e^{-2 \pi i \langle v, \xi \rangle}.
\end{equation}
However, triangulation can be rather complicated. We propose a different method using polynomial interpolation to compute $\hat{1}_K$ for non-simplicial $K$, which can be very efficient given some assumptions on the combinatorics of $K$.

    To state our Main Theorem, we need a generalization of the cross product to higher dimensions. Given $(d-1)$ vectors $x_1, \dots, x_{d-1}$ in $\mathbb{R}^d$, the \textbf{generalized cross product} $[x_1, \dots, x_{d-1}]$ is defined such that
\begin{equation}\label{eqdef:gen cross}
	\langle  [x_1, \dots, x_{d-1}], x \rangle = \det(x_1, \dots, x_{d-1},x),
\end{equation}
for any $x \in \mathbb{R}^d$. Indeed, the right-hand side gives rise to a linear functional (with variable $x$) on $\mathbb{R}^d$, which corresponds via the standard inner product to a unique vector that is taken to be the generalized cross product.

\begin{main} \label{thm:MAIN}
    Let $D = \{w_{i_1}, \dots, w_{i_{d-1}}\}$ be a diagonal of $K$. Set $D^* := [w_{i_1}, \dots, w_{i_{d-1}}]$
    \begin{itemize}
    \item[(i)] If $D$ is extremal, then 
      \begin{equation}\label{eq:pK1}
          p_K (D^*) = \varepsilon \prod_{\substack{j=1 \\ j \notin D}}^n \det(w_{i_1}, \dots, w_{i_{d-1}}, w_j),
      \end{equation}
      where $\varepsilon$ is the sign of any of the determinants on the right-hand side.
    \item[(ii)] Otherwise, if $D$ is interior, we have a simple identity:
    \begin{equation}\label{eq:pK2}
    	p_K (D^*) = 0.
    \end{equation}
    \end{itemize} 
\end{main}

In effect, the \hyperref[thm:MAIN]{Main Theorem} gives the values of the homogeneous polynomial $p_K$ at a lot of points. Therefore, if the combinatorics of the cone $K$ is sufficiently generic, we can use interpolation to determine $p_K$ exactly. This gives a novel way to compute the conic Fourier transform $\hat{1}_K$.


Suppose $p_K(\xi) = \sum_E c_E \xi^E$, where $E$ varies in the set $\mathcal{E}_{d,n} = \{(e_1, \dots, e_d) : e_i \geq 0, \sum_i e_i = n-d \}$ and $\xi^E = \xi_1^{e_1} \dots \xi_d^{e_d}$. Let us write $c = \{c_E\} \in \mathbb{R}^{\binom{n-1}{d-1}}$. Then, Equations \ref{eq:pK1} and \ref{eq:pK2} yield an (overdetermined) linear system 
\begin{equation}\label{eq:system}
	A_K x = b_K,
\end{equation} 
that has $x = c$ as a solution. Here $A_K$ is a $\binom{n}{d-1} \times \binom{n-1}{d-1}$-matrix and $b_K$ is a vector of dimension $\binom{n}{d-1}$. 

\begin{thm}\label{thm:generic}
	Suppose the generators $w_1, \dots, w_n$ of $K$ are in general positions such that no $d$-subset of them is linear dependent; or equivalently, $K$ is the cone over a simplicial polytope of dimension $d-1$.  Then, $A_K$ has full rank.
\end{thm}

Therefore, when $K$ is the cone over a simplicial polytope, we can solve System \ref{eq:system} for $c$, which in turn determines the conic Fourier transform $\hat{1}_K$. Since random points are almost surely in general positions, we believe our results will have ramifications in the theory of random polytopes.

\subsection*{Acknowledgement}
The author benefits a lot from conversations with Jamie Pommersheim and Sinai Robins. He thanks ICERM and its wonderful staff for their hospitality during part of this project. He would like to express his gratitude to Karim Adiprasito, whose grant \textit{ERC StG 716424 - CASe} supports this work.

\section{Evaluations at diagonals}
Suppose $K \subset \mathbb{R}^d$ has apex $v$ and $W = \{w_1, \dots, w_n\}$ is the set of its generators. More concretely, 
\begin{equation}
	K = \{v + t_1 w_1 + \dots + t_n w_n : t_1, \dots, t_n > 0\}.
\end{equation}
We use $W^{[k]}$ to denote the set of all $k$-subsets of $W$; hence, $W^{[d-1]}$ is the set of diagonals of $K$. For any subset $S$ of $\mathbb{R}^d$, we denote
\begin{align*}
	\lin(S) &:= \{\sum_{x \in S} t_x x \text{ with } t_x \in \mathbb{R}\}, \\
	\cone(S) &:= \{\sum_{x \in S} t_x x \text{ with } t_x > 0\}, \\
	\conv(S) &:= \{\sum_{x \in S} t_x x \text{ with } 0 < t_x < 1\}. \\
\end{align*}

We note that, using the standard inner product on $\mathbb{R}^d$, one can conceptually think of the generalized cross product $[x_1, \dots, x_{d-1}]$ either as the projective dual of the hyperplane $\lin(D)$, or as the Hodge dual of the wedge product $x_1 \wedge \dots \wedge x_{d-1}$. Computationally, the generalized cross product can be calculated by the Gram--Schmidt process, or more straightforwardly, by taking the $(d-1)$-minors of the $d \times (d-1)$-matrix formed by collating the vectors $x_1, \dots, x_{d-1}$, as in the definition of the $3$-dimensional cross product.

\begin{proof}[Proof of \texorpdfstring{\hyperref[thm:MAIN]{Main Theorem}}{Main Theorem}]
	First of all, because the Fourier transform converts a translation (i.e. time shifting) into a modulation (i.e. frequency shifting), we have
\begin{equation}
	\hat{1}_{S+x_0} (\xi) = e^{-2 \pi i \langle x_0, \xi \rangle} \hat{1}_S(\xi).
\end{equation}
Thus, we may assume that the apex $v$ of $K$ is at the origin $0$.

Given a diagonal $D = \{w_{i_1}, \dots, w_{i_{d-1}}\}$, recall that $D^* = [w_{i_1}, \dots, w_{i_{d-1}}]$. We write $\det(D, x) := \det(w_{i_1}, \dots, w_{i_{d-1}}, x) = \langle D^*, x \rangle$, for brevity. Clearly, $\langle D^*, w_{i_1} \rangle = \dots = \langle D^*, w_{i_{d-1}} \rangle = 0$.
    
    \medskip
    
	\emph{For Case (i)}, take a triangulation $K = K_1 \cup \dots \cup K_m$ such that one of the simplicial cones, say $K_1$, generated by $D$, together with an extra vector $w_{i_d}$. Since $K_1$ is simplicial, Equation \ref{eq:simplicial} gives
    \begin{equation}
    	\hat{1}_{K_1}(\xi) = \frac{|\det(D, w_{i_d})|}{\prod_{k=1}^d \langle \xi, w_{i_k} \rangle} = \frac{|\det(D, w_{i_d})| \prod_{\substack{1 \leq j \leq n \\ j \notin D, j \neq w_{i_d}}} \langle \xi, w_{j}\rangle}{\prod_{j=1}^n \langle \xi, w_{j} \rangle}
    \end{equation}
    The numerator of the last fraction, evaluated at $\xi=D^*$, is equal to 
    \begin{equation}\label{eq:det1}
    	\varepsilon \prod_{\substack{1 \leq j \leq n \\ j \notin D}} \det(D, w_j),
    \end{equation} 
    where $\varepsilon$ is the sign of $\det(D, w_{i_d}) = \det(w_{i_1}, \dots, w_{i_d})$. We note that this sign does not depend on the triangulation. Indeed, because $K$ is convex, all other generators lie on one side of the hyperplane $\lin(D)$. Thus, all determinants in Equation \ref{eq:det1} have the same sign, which implies the sign $\varepsilon$ is independent of the triangulation.
		
	Take $p \geq 2$. Since $D$ is extremal, the simplicial cone $K_p$ cannot contain $D$. Therefore, when we equate the denominator of the conic Fourier transform $\hat{1}_{K_p}$ to $\prod_{j=1}^n \langle \xi, w_{j} \rangle$, the numerator must contain one of the factors $\langle \xi, w_{i_1} \rangle, \dots, \langle \xi, w_{i_{d-1}} \rangle$, which vanish when evaluated at $\xi = D^*$. This completes our proof in Case (i).
    
    \medskip
		
	\emph{For Case (ii)}, take a triangulation $K = K_1 \cup \dots \cup K_m$ such that two of the simplicial cones, say $K_1$ and $K_2$, have $w_{i_1}, \dots, w_{i_{d-1}}$ as generators. Let $w_{i^1_d}$ and $w_{i^2_d}$ be the other generators in $K_1$ and $K_2$, respectively. We have
		
	\begin{align} 
        \hat{1}_{K_1}(\xi) &= \frac{|\det(D, w_{i^1_d})|}{\langle \xi, w_{i^1_d} \rangle \prod_{k=1}^{d-1} \langle \xi, w_{i_k} \rangle} \nonumber \\
        	&= \frac{\langle \xi, w_{i^2_d} \rangle |\det(D, w_{i^1_d})| \cdot \prod_{\substack{1 \leq j \leq n \\ j \notin D \cup \{i^1_d,i^2_d\}}} \langle \xi, w_{j}\rangle}{\prod_{j=1}^n \langle \xi, w_{j} \rangle} \label{eq:k1 opposite}
	\end{align}
    
    \begin{align}
        \hat{1}_{K_2}(\xi) &= \frac{|\det(D, w_{i^2_d})|}{\langle \xi, w_{i^2_d} \rangle \prod_{k=1}^{d-1} \langle \xi, w_{i_k} \rangle} \nonumber \\
        	&= \frac{\langle \xi, w_{i^1_d} \rangle |\det(D, w_{i^2_d})| \cdot \prod_{\substack{1 \leq j \leq n \\ j \notin D \cup \{i^1_d,i^2_d\}}} \langle \xi, w_{j}\rangle}{\prod_{j=1}^n \langle \xi, w_{j} \rangle} \label{eq:k2 opposite}
	\end{align}
		
		Observe that, because $K_1$ and $K_2$ have disjoint interiors, two vectors $w_{i^1_d}$ and $w_{i^2_d}$ lie on the opposite sides of the hyperplane $\lin(D)$. Therefore, $\det(D, w_{i^1_d})$ and $\det(D, w_{i^2_d})$ have opposite signs, which implies
	\begin{align}
		& \langle D^*, w_{i^2_d} \rangle \cdot |\det(D, w_{i^1_d})| + \langle D^*, w_{i^1_d} \rangle \cdot |\det(D, w_{i^2_d})| \nonumber \\ 
        =\ & \det(D, w_{i^2_d}) \cdot |\det(D, w_{i^1_d})| + \det(D, w_{i^1_d}) \cdot |\det(D, w_{i^2_d})| \nonumber \\
        =\ & 0.
	\end{align}
    Hence, the numerators of the fractions in Equations \ref{eq:k1 opposite} and \ref{eq:k2 opposite}, when evaluated at $\xi = D^* = [w_{i_1}, \dots, w_{i_{d-1}}]$, sum up to $0$.
		
	Let $j \geq 3$. Since both sides of $\cone(D)$ have been covered by $K_1$ and $K_2$, the cone $K_j$ cannot contain $D$. Using the same argument as in Case (i), we see that the contribution of $K_j$ ($j \geq 3$) to the polynomial $p_K$ (of Equation \ref{eqdef:pK}) evaluates to $0$ when $\xi=D^*$. Thus, overall, we have proved $p_K(D^*) = 0$ and finished the proof in Case (ii).
\end{proof}

\section{Veronese--Vandermonde determinants}
The rows of the matrix $A_K$ in System \ref{eq:system} remind us of the Veronese map and the dependence of $A_K$ on the positions of the generators $w_i$ of the cone $K$ resembles the Vandermonde determinant. In this section, we will flesh out these connections.

The classical Vandermonde determinant is 
\begin{equation}
	\det \left(x_i^{k-1} \right)_{i,k = 1}^m = \prod_{1 \leq i < j \leq m} (x_j - x_i).
\end{equation} 

We would like to generalize this identity to higher dimensions using the Veronese map. Recall that the Veronese map $\nu_k:\mathbb{R}^d \to \mathbb{R}^{\binom{d+k-1}{d-1}}$ of degree $k$ is defined by taking all monomials of degree $k$ on $d$ variables and then evaluating all these monomials at each point of $\mathbb{R}^d$. For instance,
\[ \nu_2 (x_1, x_2, x_3) = (x_1^2, x_1 x_2, x_1 x_3, x_2^2, x_2 x_3, x_3^2). \]
The $i$-th row of the Vandermonde matrix can be thought of as the image of the point $(x_i,1) \in \mathbb{R}^2$ under the Veronese map $\nu_{m-1}$. Similarly, in System \ref{eq:system}, each row of the $\binom{n}{d-1} \times \binom{n-1}{d-1}$-matrix $A_K$ equals to $\nu_{n-d}(D^*)$, for a diagonal $D$ of $K$.

A maximal minor (of order $\binom{n-1}{d-1}$) of $A_K$ is given by a choice of $\binom{n-1}{d-1}$ diagonals of $K$, i.e. by an element $\mathcal{D} \in \left(W^{[d-1]}\right)^{[\binom{n-1}{d-1}]}$. We write $\mu_\mathcal{D}(A_K)$ to be the maximal minor of $A_K$ associated to $\mathcal{D}$.

If we consider elements of $W^{[p]}$ as simplices of dimension $p-1$, we can think of elements of $\left(W^{[p]}\right)^{[q]}$ as subcomplexes of the complete simplicial complex $\Sigma^{p-1}(W)$ of dimension $p-1$ with vertices in $W$. Therefore, we can regard $\mathcal{D}$ as a subcomplex of $\Sigma^{d-2}(W)$. 

Given a $(d-1)$-simplex $E \in W^{[d]}$, we define the \textbf{multiplicity} $\mult_\mathcal{D} (E)$ of the complex $\mathcal{D}$ at $E$ to be the number of facets of $E$ contained in $\mathcal{D}$, or equivalently, the number of elements $D$ of $\mathcal{D}$ such that $D \subset E$. We say that the subcomplex $\mathcal{D}$ \textbf{fills} $\Sigma^{d-2}(W)$ if, for any $(d-1)$-simplex $E$ in $\Sigma^{d-1}(W)$, at least one facet of $E$ belongs to $D$. In other words, $\mathcal{D}$ fills $\Sigma^{d-2}(W)$ if and only if $\mult_\mathcal{D} (E) \geq 1$ for all $E \in W^{[d]}$. 

\begin{prop}[Veronese--Vandermonde determinants] \label{prop:VerVan}
	Suppose the generators $w_i$ of $K$ are in general positions. Let $\mathcal{D}$ be a family of $\binom{n-1}{d-1}$ diagonals of $K$. 
    \begin{itemize}
    \item[(i)] If $\mathcal{D}$ does not fill the complete simplicial complex  $\Sigma^{d-2}(W)$, then
    \begin{equation} \label{eq:VerVan1}
    	\mu_\mathcal{D} (A_K) = 0.
    \end{equation}
    \item[(ii)] If $\mathcal{D}$ fills $\Sigma^{d-2}(W)$, then the minor of $A_K$ associated to $\mathcal{D}$ is
    \begin{equation} \label{eq:VerVan2}
    	\mu_\mathcal{D} (A_K) = \pm \prod_{E \in W^{[d]}} \det(E)^{\mult_\mathcal{D} (E) - 1}.
    \end{equation}
    \end{itemize} 
\end{prop}

\begin{proof}[Proof of Theorem \ref{thm:generic}]
	Because we can always choose a family $\mathcal{D}$ that fills the complete simplicial complex  $\Sigma^{d-2}(W)$, Proposition \ref{prop:VerVan} implies that $A_K$ has a nonzero maximal minor, and thus, it has full rank.
\end{proof}

This proposition seems related to Theorem 4.15 in Ben Yaacov \cite{yaacov2014multivariate}, but we have not been able to figure out the connection. Since we do not need such generality as in \cite{yaacov2014multivariate}, we will provide an elementary proof resembling that of the Vandermonde determinant.

\begin{proof}[Proof of Proposition \ref{prop:VerVan}]

    \textit{In Case (i),} we will prove the matrix of $\mu_\mathcal{D} (A_K)$ admits a nonzero null vector. Because the subcomplex $\mathcal{D}$ does not fill $\Sigma^{d-2}(W)$, there exists a $(d-1)$-simplex $E \in W^{[d]}$ such that $D \not\subset E$ for all $D \in \mathcal{D}$. Therefore, any $D$ in $\mathcal{D}$ must intersect $F = W \setminus E$. Note that each row of $\mu_\mathcal{D} (A_K)$ is of the form $\nu_{n-d}(D^*)$ for some $D \in \mathcal{D}$. In Lemma \ref{lem:null2}, we construct from $F$ a nonzero vector $F^+$ such that $\langle F^+, \nu_{n-d}(D^*) = 0$, which means the matrix of $\mu_\mathcal{D} (A_K)$ admits a nonzero null vector. Therefore, $\mu_\mathcal{D} (A_K) = 0$, as desired.
    
    \textit{In Case (ii),} we think of $W$ as a matrix of $nd$ variables $w_{ij}$ for $1 \leq i \leq n$ and $1 \leq j \leq d$. Since the generalized cross product can be computed by taking minors, the entries of $\mu_\mathcal{D} (A_K)$ are polynomials in $w_{ij}$, which implies that the determinant $\mu_\mathcal{D} (A_K)$ is also a polynomial in $w_{ij}$.
     
    As a sanity check, let us compare the degrees of the two sides of Equation \ref{eq:VerVan2} as polynomials in $w_{ij}$. For each $D \in \mathcal{D}$, the coordinates of $D^*$ are determinants of $(d-1) \times (d-1)$-matrices and have degree $(d-1)$. The application of the Veronese map $\nu_{n-d}$ raises the degrees to $(d-1)(n-d)$. Hence, the matrix of $\mu_\mathcal{D} (A_K)$ have entries with degrees $(d-1)(n-d)$ and its size is $\binom{n-1}{d-1}$. Therefore, the total degree of the left-hand side of Equation \ref{eq:VerVan2} is $(d-1)(n-d)\binom{n-1}{d-1}$. On the other hand, the total degree of the right-hand side is
   	\begin{align*}
   		&d \left( \sum_{E \in W^{[d]}} mult_\mathcal{D}(E) - \binom{n}{d} \right) = d \left( \sum_{D \in \mathcal{D}} (n - \# D) - \binom{n}{d} \right) \\
        =\quad & d \left( \binom{n-1}{d-1} (n-d+1) - \binom{n}{d} \right) = \binom{n-1}{d-1} (d(n-d+1) - n) \\
        =\quad & \binom{n-1}{d-1} (d-1) (n-d),
   	\end{align*} 
as expected.

    For $D = \{w_{i_1}, \dots, w_{i_{d-1}}\}$, the defining equation (Equation \ref{eqdef:gen cross}) of the generalized cross product infers that, if the vectors $D$ are dependent, then $D^* = 0$. Also, if $w_j$ is dependent on $D$ and $D'$ is obtained by replacing a vector in $D$ with $w_j$, then $D'^*$ is parallel to $D$, that is $D'^* = \lambda D^*$ for some $\lambda \in \mathbb{R}$. Therefore, $\nu_{n-d} D'^* = \lambda^{n-d} \nu_{n-d} D^*$ are parallel.
    
    Let $E \in W^{[d]}$ be a $(d-1)$-simplex. Suppose $\mult_\mathcal{D} (E) \geq 2$. Let $D_1, D_2 \in \mathcal{D}$ such that $D_1, D_2 \subset E$. Without loss of generality, let us say $D_1 = \{w_1, ..., w_{d-2}, w_{j_1}\}$ and $D_2 = \{w_1, ..., w_{d-2}, w_{j_2}\}$. If $\det(E) = 0$, then $w_{j_2}$ is dependent on $D_1$. Therefore, as noted above, $D_1^*$ and $D_2^*$ are parallel, and so are $\nu_{n-d} D'^*$ and $\nu_{n-d} D^*$. Because of the fact that multivariate polynomial rings over fields are unique factorization domains, we see that the polynomial $\det(E)$ divides $\mu_\mathcal{D} (A_K)$. 
    
    Suppose $D_1, \dots, D_k$ are all the diagonals in $\mathcal{D}$ that are contained in the simplex $E$, where $k = \mult_\mathcal{D} (E)$. Then, the number of appearances of $\det(E)$ as divisors of $\mu_\mathcal{D} (A_K)$ is equal to $\mult_\mathcal{D} (E) - 1$. This can be seen by fixing $D_1$ and letting the $(k-1)$ vectors $D_2 \setminus D_1, \dots D_k \setminus D_1$ independently approach the hyperplane $\lin(D_1)$.

	Thus far, we have proved that the RHS of Equation \ref{eq:VerVan2} divides $\mu_\mathcal{D}(A_K)$. By the sanity check above, the degrees of the two sides are equal, which infers that they differ only by a constant factor. Since the coefficients of $\mu_\mathcal{D}(A_K)$ and $\det(E)$ are $\pm 1$, the constant factor is also $\pm 1$. Therefore, we have completed the proof in Case (ii).
    
\end{proof}

We need the following lemma before stating Lemma \ref{lem:null2}.
\begin{lem}\label{lem:null1}
    	Take $n \geq d > 0$. Let $P$ be a $\binom{n-1}{d-1} \times d$-matrix whose rows are indexed by $IV_{d,n-d}$, where $IV_{d,s} := \{ x \in \mathbb{Z}^d_{\geq 0}: x_1 + \dots + x_d = s \}$, such that
        \begin{itemize}
        \item If $x_j = 0$, then the $(x,j)$-entry of $P$ is zero.
        \item For any $y \in IV_{d,n-d-1}$, the $(y + e_j, j)$-entries, for $1 \leq j \leq d$ are equal, say, to a number $c_y$. Here $e_j$ is the $j$-th standard vector of dimension $d$.
        \end{itemize}
        Let $Q$ be an antisymmetric $d \times d$-matrix. Then, for any $v \in \mathbb{R}^d$,
        \begin{equation}\label{eq:null1}
        	\langle Pv, \nu_{n-d}(Qv)  \rangle = 0.
        \end{equation}
        Recall that here $\nu_{n-d}$ is the Veronese map of degree $n-d$.
    \end{lem}
    \begin{proof}
    	Note that the rows of $\nu_{n-d}(Qv)$ are indexed by $IV_{d,n-d}$ and equal to $(Qv)^x$ for $x \in IV_{d,n-d}$, where $a^x = a_1^{x_1} \dots a_d^{x_d}$. We can see that
        \begin{equation}
        	\langle Pv, \nu_{n-d}(Qv)  \rangle = \sum_{y \in IV_{d,n-d-1}} c_y (v^T Q v).
        \end{equation}
        But, $v^T Q v = 0$ because $Q$ is antisymmetric. This proves the lemma.
    \end{proof}
    
    \begin{lem}\label{lem:null2}
    	Let $F \subset \mathbb{R}^d$ be a $(n-d)$-subset and $D \subset \mathbb{R}^d$ a $(d-1)$-subset such that $D$ and $F$ intersect. Then, we can construct from $F$ a nonzero vector $F^+$ of dimension $\binom{n-1}{d-1}$ such that 
        \begin{equation}\label{eq:null2}
        	\langle F^+, \nu_{n-d}(D^*) \rangle = 0.
        \end{equation}
    \end{lem}
    \begin{proof}
    	Suppose $F = \{v_1, \dots, v_m\}$. We will construct $F^+$ whose coordinates are indexed by $x \in IV_{d,n-d}$. Then, we set the $x$-th coordinate of $F^+$ to be
        \begin{equation}
        	F^+_x = \sum_{\substack{k \in \{1,\dots,d\}^{n-d} \\ \sigma(k) = x}} \prod_{i=1}^{n-d} (v_i)_{k_i},
        \end{equation}
        where $\sigma(k) = \sum_{i=1}^{n-d} k_i e_i$ with $e_i$ the $i$-th standard vector of dimension $d$. We can check that $F^+$ has the form $Pv_1$ for a matrix $P$ as in Lemma \ref{lem:null1}.
        
        Suppose $v_1 \in D \cap F$. Then, $D^* = Qv_1$ for some antisymmetric $d \times d$-matrix. Therefore, by Lemma \ref{lem:null1},
		\begin{equation}
			\langle F^+, \nu_{n-d}(D^*) \rangle = \langle Pv_1, \nu_{n-d}(Qv_1)  \rangle = 0.
		\end{equation}        
    \end{proof}

\bibliographystyle{alpha}
\bibliography{sample}

\end{document}